\documentclass[12pt]{amsart}
\usepackage{amsmath,bbm, color}
\usepackage{latexsym,mathrsfs,bm}
\usepackage{url}
\usepackage{paralist}
\usepackage{amsthm,amsfonts,amssymb,amsmath}
\usepackage{comment}
\usepackage{hyperref}
\usepackage[latin1]{inputenc}

\newcommand{\ee}{e}

\newcommand{\im}{\operatorname{Im}}
\newcommand{\Mod}[1]{\ (\text{mod}\ #1)}

\newcommand{\R}{\mathbb{R}}

\newcommand{\C}{\mathbb{C}}

\newcommand{\Z}{\mathbb{Z}}

\newtheorem{thm}{Theorem}[section]

\newtheorem{lemma}[thm]{Lemma}

\theoremstyle{remark}

\newtheorem{rem*}{Remark}

\newcommand{\sgn}{\operatorname{sgn}}

\newcommand\srel[2]{\begin{smallmatrix} {#1} \\ {#2} \end{smallmatrix}}

\let\originalleft\left
\let\originalright\right
\renewcommand{\left}{\mathopen{}\mathclose\bgroup\originalleft}
\renewcommand{\right}{\aftergroup\egroup\originalright}

\numberwithin{equation}{section}

\begin{document}
\title[The $D^6 R^4$ interaction as a Poincar\'e series, and a  shifted convolution sum]{The $D^6 R^4$ interaction as a Poincar\'e series, and a related shifted convolution sum}
\author{Kim Klinger-Logan, Stephen D. Miller, and Danylo Radchenko}

\maketitle
\date{\today}

\begin{abstract}
We complete the program, initiated in a 2015 paper of Green, Miller, and Vanhove, of directly constructing the automorphic solution to the string theory $D^6 R^4$ differential equation $(\Delta-12)f=-E_{3/2}^2$ for $SL(2,\Z)$.  The construction is via a type of  Poincar\'e series, and requires explicitly evaluating a particular double integral.  We also show how to use double Dirichlet series to formally derive the predicted vanishing of one type of term appearing in $f$'s Fourier expansion, confirming a conjecture made by Chester, Green, Pufu, Wang, and Wen motivated by Yang-Mills theory (and later proved rigorously by Fedosova, Klinger-Logan, and Radchenko using the Gross-Zagier  Holomorphic Projection Lemma.).   
\end{abstract}

\section{Introduction}

Many quantities of interest in type IIB string theory are naturally automorphic functions on an arithmetic quotient of a symmetric space, such as $SL(2,\Z)\backslash SL(2,\R)/SO(2)$.  For example,
the coefficients of the low energy expansion of the 4-graviton scattering amplitude are functions on the coset space $G/K$, where $G$ denotes the real points of the (algebraic) duality group and $K$ is a maximal compact subgroup of $G$.  Furthermore, $U$-duality   implies that the scattering amplitude should obey transformation rules under a discrete subgroup $\Gamma$ of $G$. In 10 space-time dimensions, for example, $G=SL(2,\mathbb R)$ and the coefficients in the low energy expansion are modular functions. Furthermore, BPS symmetry implies some of the lower order coefficients are in fact eigenfunctions of the laplacian, making them automorphic forms.  By now many examples are known and have been well-studied, e.g., \cite{BKP2020,CGPWW2020,DK2019,DKS2022a,DKS2022b,GMRV2010,GMV2015,GRV2010}. 

A good example is the $R^4$ term in the low energy expansion of  maximally supersymmetric string theory. In 10 dimensions, the coefficient function for this term is given by
\begin{align}\mathcal{E}_{(1,0)}(z) &= 2\zeta(3)E_{3/2}(z)\,,
\end{align}
where  $E_s(z)$ is the non-holomorphic Eisenstein series
\begin{equation}E_{s}(z)=\sum_{\gamma\in (B\cap \Gamma)\backslash \Gamma} \text{Im}(\gamma z)^{s}, \quad \text{Re}(s)>1\,,
\end{equation}
for  $z=x+iy$ in the complex upper half plane $\mathbb{H}\cong SL(2,\mathbb{R})/ SO(2)$, where $\Gamma=SL(2,\Z)$ and $B$ is the subgroup of upper triangular matrices in $SL(2,\mathbb{R})$.
In this case, $E_{s}(z)$ is an eigenfunction of the laplacian $\Delta=y^2\left(\partial_x^2+\partial_y^2\right)$, and hence the automorphic function $\mathcal{E}_{(1,0)}$ is furthermore an automorphic form.

However, BPS protection only carries so far and some of the most intriguing and intricate coefficients are not automorphic forms themselves, but instead related to them in a nonlinear way.
In this paper, we will revisit the $D^6 R^4$ interaction given by the $\mathcal{E}_{(0,1)}$ term studied in \cite{DGH2015, GHV1999,GMV2015,GV2006,KKL2018,Pioline}, which mathematically amounts to the solution to the inhomogeneous automorphic differential equation
\begin{equation}\label{eq:DE}
(\Delta - 12) f(z) = -\left(2\zeta(3)E_{3/2}(z)\right)^2
\end{equation}
satisfying certain growth conditions.
In \cite{GMV2015}, Green, Miller, and Vanhove give the full Fourier expansion of the solution $f(z)$. However, this solution, though explicit, was not computed directly -- rather, it was correctly guessed and subsequently verified. In fact, the appendices of \cite{GMV2015} go on to describe three potential approaches to solving (\ref{eq:DE}) directly, all of which ran into difficulties attempting to extract  the full Fourier expansion of $f(z)$.

Given the abundance of automorphic differential equations generalizing (\ref{eq:DE}) (see \cite{AK2018, CGPWW2020, DK2019, DKS2022a,DKS2022b, GKS2020}) it is desirable to be able construct a solution via a {\it method}, rather than to rely on guesswork.
In what follows, we resolve the issue that arose
in the first of those three approaches, namely the Poincar\'e series construction in    \cite[Appendix A]{GMV2015}.  The obstacle there to obtaining the explicit form of the Fourier expansion for $f$ was the difficulty of computing the integral presented in (A.40) there, i.e.,   the $a=3/2$ case of  \eqref{eq:integral} below. In Section \ref{Sec:sol}, we generalize this  Poincar\'{e} series method to equations of the form
\begin{align}(\Delta-\lambda)f(z)=-E_a(z)^2\,.
\end{align}
Specifically, the methods we outline are applicable to the cases when $a$ is a half-integer or even integer, and with specific choices of eigenvalue $\lambda$. In Section \ref{sec:int}, we explicitly compute the obstructing integral \cite[(A.40)]{GMV2015} (equivalently, \eqref{eq:integral})  which arose in the case of $a=3/2$ and $\lambda=12$. The method outlined in Section \ref{sec:int} extends (at least experimentally) to the Poisson summation integrals that arise for other half-integer values for $a$.

The computation   in Section \ref{sec:int} yields the same Fourier expansion found in \cite{GMV2015}:
  \begin{equation}f(z) = \sum_{n \in \Z}  f_n(y) e^{2 \pi i n x}\,,\end{equation}
	where
	\begin{align}\label{eq:GMVsol}
	f_n(y)& = \delta_{n,0}\tilde f(y) + \alpha_n\sqrt{y} K_{7/2}(2\pi |n|y)\\
	& \ \ \ \ +\sum_{\substack{n_1+n_2=n\\ (n_1,n_2)\neq (0,0)}}\sum_{i,j=0,1}M_{n_1,n_2}^{i,j}(\pi |n|y) K_i(2\pi |n_1|y) K_j(2\pi |n_2|y)\,,\nonumber
	\end{align}
	in which $\tilde f, M_{n_1,n_2}^{i,j},$ and  $\alpha_n$ are polynomials in $y$ and $\frac{1}{y}$  as defined in \cite[Section 2.2]{GMV2015}.
	Explicitly, each  $\alpha_n$ can be expressed as
	\begin{equation}\label{eq:alpha}\alpha_n=\sum_{n_1+n_2=n}\alpha_{n_1,n_2}\,,\end{equation}
	where
	$\alpha_{0,0}=0$,
	\begin{equation}\label{eq:alpha1}\displaystyle \lim_{n_2\to -n_1} \alpha_{n_1,n_2}\sqrt{y} K_{7/2}(2\pi |n_1+n_2|y)=\frac{8\sigma_2(|n_1|)^2  }{21n_1^6\pi^2 y^3}\quad \text{ for } n_1\neq 0,\end{equation}
\begin{align}\label{eq:alpha2}\alpha_{n,0}=\alpha_{0,n}= \frac{64\sigma_2(|n|)\left(n^2\pi^4-90\zeta(3)\right)}{135|n|^{5/2}\pi}\quad \text{ for }  n\neq 0,\end{align}
and
		\begin{align}\label{eq:alpha3}\alpha_{n_1,n_2}=\sgn(n_1+n_2)&\frac{128\pi \sigma_2(|n_1|)\sigma_2(|n_2|)}{45n_1^2n_2^2|n_1+n_2|^{7/2}}
	\left( \vphantom{\sum} n_1^5+n_2^5+15n_1^4n_2+15n_1n_2^4\right.\\  & \ \ \ \left.-80n_1^3n_2^2-80n_1^2n_2^3+60n_1^2n_2^2(n_1-n_2)\log(\textstyle{|\frac{n_1}{n_2}|} )\right)
	\nonumber\end{align}
for $n\neq 0$ and $n_1n_2\neq 0$, where $\sigma_k(n)=\sum_{d|n}d^k$ is the divisor sum function.

 The individual $\alpha_{n_1,n_2}$ arise  from a homogeneous solution to a differential equation, or from the integral \cite[(A.40)]{GMV2015} as computed in this paper; however, the sum (\ref{eq:alpha}) is of greater interest.
	Based on ideas from the AdS-CFT correspondence and Yang-Mills theory, Chester, Green, Pufu, Wang, and Wen made the surprising conjecture  in \cite[Section C.1(a)]{CGPWW2021} that   \begin{equation}\label{eq:conj2}\alpha_n  \ = \  \sum_{n_1+n_2=n}\alpha_{n_1,n_2} \ = \ 0\end{equation}
for each $n\neq 0$.
In other words, they conjectured that the total sum of the Fourier coefficients in \eqref{eq:alpha} corresponding to the homogeneous solution to (\ref{eq:DE}) {\it vanishes}.  This is not at all apparent from the formulas above, from which even the convergence of the sum  is not manifest.
In Section \ref{sec:Danylo} we provide a formal proof of this vanishing conjecture of Chester, Green, Pufu, Wang, and Wen, based on formal (but nonrigorous) manipulations of Dirichlet series.  After this paper was first circulated, the full vanishing conjecture was generalized and rigorously proven by Fedosova and the first and third named authors via other means (see Section~\ref{sec:Danylo} for further comments).

Since this paper was first released there have been a number of follow up works on the material in Section~\ref{sec:Danylo} (see \cite{FKL1,FKL2,FKLR}).  Perhaps most notably, the paper \cite{FKLR} improves on our results by giving a family of identities which relate convolutions of divisor functions to Fourier coefficients of modular forms. In particular, it  now gives a rigorous proof of the conjecture of Chester, Green, Pufu, Wang, and Wen using the Gross-Zagier Holomorphic Projection Lemma.

\section{Solution to $(\Delta-\lambda)f=-E_a^2$}\label{Sec:sol}
In this section we outline the Poincar\'e series method for     finding solutions of the automorphic differential equation
\begin{align}\label{eq:DE1}(\Delta-\lambda)f=-E_a^2
\end{align}
for special values of $a>1$ and $\lambda>0$, recapping \cite[Appendix~A]{GMV2015}.
Using the definition of the Eisenstein series, one expands
\begin{equation}\label{eq:E_a^2}
E_{a}(z)^2
\ =  \sum_{\srel{(m_1,n_1)\in \mathcal{S}}{(m_2,n_2)\in \mathcal{S}}} |m_1n_2-n_1m_2|^{-2a}\,\mathcal{T}\left(\frac{m_1z+n_1}{m_2z+n_2}\right),
\end{equation}
 where
$$\mathcal{S}:=\{(0,1)\}\sqcup\{(c,d)\in\Z\times\Z~|~ c>0\,\&\,\gcd(c,d)=1\},$$
and
$$\mathcal{T}(x+iy) =\sigma(x/y) := ((x/y) ^2+1)^{-a}$$  (terms  in \eqref{eq:E_a^2} with $m_1n_2=n_1m_2$ are replaced by an appropriate limit).
This is a consequence of the identities
\begin{align*}
\frac{\text{Re}(\gamma z)}{\text{Im}(\gamma z)}=\frac{n_1n_2+m_2n_1 x+m_1n_2x+m_1m_2(x^2+y^2)}{y(\det\gamma)}
\end{align*}
and
\begin{align*}
(n_1n_2+m_2n_1x+m_1n_2x+m_1m_2(x^2+y^2))^2+(y(\text{det}\gamma))^2
= |m_1z+n_1|^2|m_2z+n_2|^2\,,
\end{align*}
where  $\gamma=\begin{bmatrix} m_1 &n_1\\ m_2 & n_2\end{bmatrix}$.

In terms of $u=x/y$ the differential equation (\ref{eq:DE1}) transforms to
\begin{align}\label{eq:DE2}\left(\frac{d}{du}\left((1+u^2)\frac{d}{du} \right)-\lambda \right)h_{a,\lambda}(u)  \ = \ -\sigma(u) \ = \  -(u ^2+1)^{-a}\,.
\end{align}
This differential equation can be solved explicitly in many cases, including that of $a=3/2$ and $\lambda=12$ originally studied by \cite{GMV2015} -- see (\ref{eq:h3/2}) below.  In general, it has a 2-dimensional family of solutions, and frequently a unique solution   satisfying  the decay condition that  $h_{a,\lambda}(u)= O(|u|^{-2a})$  for large values of $u=x/y$.  This latter point can be checked by asymptotically solving the differential equation, and will be tacitly assumed in the rest of the paper (in particular, it was shown in \cite{GMV2015} to hold in the case of $a=3/2$ and $\lambda=12$).
In terms of that solution we define $F(x+iy):= h_{a,\lambda}(x/y)$, so that $(\Delta- \lambda) F = -\mathcal{T}$.  Summing over $\mathcal{S}$, the solution to (\ref{eq:DE1})
is given by the sum
\begin{align}\label{eq:Poincaresol}
f(z) =  \sum_{\srel{(m_1,n_1)\in \mathcal S}{(m_2,n_2)\in \mathcal{S}}} |m_1n_2-n_1m_2|^{-2a}\,F\left(\frac{m_1z+n_1}{m_2z+n_2}\right),
\end{align}
which is absolutely convergent since $h_{a,\lambda}$ is bounded by a constant multiple of $\sigma$, hence $F$ is bounded by a constant multiple of $\mathcal{T}$, and the sum in (\ref{eq:Poincaresol}) converges for $\text{Re}(a)>1$ because (\ref{eq:E_a^2}) does.

The Fourier modes
\begin{equation}\label{eq:Fexp}\hat f_{n}(y) \ =\sum_{n_1+n_2=n}\hat f_{n_1,n_2}(y)\end{equation} of the solution
$
f(z) = \sum_{n \in\mathbb{Z}}\hat f_{n}(y) e^{2\pi i n x}
$
can be derived as in \cite[Appendix~A]{GMV2015} by writing
$$f(z) \ = \  \Sigma^{0,0}(z) +\Sigma^{0,1}(z) + \Sigma^{1,0}(z)+ \Sigma^{1,1}(z),  $$
where
\begin{align}
 \Sigma^{0,0}(z) & =\lim_{(m_1n_2-n_1m_2)\to 0}|m_1n_2-n_1m_2|^{-2a}\,F\left(\frac{m_1z+n_1}{m_2z+n_2}\right),\label{stuff1}\\
 \Sigma^{0,1}(z) & = \sum_{m_2=1}^\infty\sum_{(n_2,m_2)=1}|m_2|^{-2a}\,F\left(\frac{1}{m_2z+n_2}\right),\label{stuff2}\\
 \Sigma^{1,0}(z) & = \sum_{m_1=1}^\infty\sum_{(n_1,m_1)=1}|m_1|^{-2a}\,F\left(m_1z+n_1\right),\label{stuff3}\\
 \Sigma^{1,1}(z)  & =  \sum_{m_1 =1}^\infty\sum_{(n_1,m_1)=1}\sum_{m_2=1}^\infty\sum_{(n_2,m_2)=1}|m_1n_2-n_1m_2|^{-2a}\,
 F\left(\frac{m_1z+n_1}{m_2z+n_2}\right)\label{stuff4}
\end{align}
(terms in \eqref{stuff1} -- \eqref{stuff4}  with $m_1n_2=n_1m_2$ are replaced with an appropriate limit),
and then applying Poisson summation. Note that the limit in $\Sigma^{0,0}$ is well-defined. In terms of the customary notation  $e(x)=e^{2\pi i x}$  it is straightforward to see that
$$ \Sigma^{0,1}(z) =  \Sigma^{1,0}(z) = \sum_{n\in\Z}\widehat \Sigma_n^{0,1}(y)e( n x)\ \ \ \text{ and } \ \ \   \Sigma^{1,1}(z)  = \sum_{n_1\in\Z}\sum_{n_2\in\Z}\widehat \Sigma_{n_1,n_2}^{1,1}(y)e\left((n_1+n_2) x\right),$$ while inputting Ramanujan sums as in  \cite[(A.28)]{GMV2015} gives the formula
\begin{align}
\Sigma^{0,1}(z) &= \sum_{m_2=1}^\infty\frac{1}{|m_2|^{2a}} \sum_{n_2\in (\Z/m_2\Z)^*} \sum_{n\in\Z}e\left(n( x+\textstyle{\frac{n_2}{m_2}})\right)y\,\widehat h_{a,\lambda}(ny)\\
& = \frac{1}{\zeta(2a)}\sum_{n\in\Z}e(n x)n^{1-2a}\sigma_{2a-1}(|n|)y\, \widehat h_{a,\lambda} (ny)\,.
\end{align}
Note that in the case of $a=3/2$ and $\lambda = 12$ studied in \cite{GMV2015}, the authors gave explicit formulas for both $h_{\frac{3}{2},12}$ and $\widehat{h}_{\frac{3}{2},12}$.

However, the term  $\Sigma^{1,1}$ is harder because double-Poisson summation gives the more complicated formula
\begin{align*}
\Sigma^{1,1}(z)&
= \sum_{\substack{ n_1, n_2\in\Z}}
\frac{\sigma_{2a-1}(| n_1|)\sigma_{2a-1}(| n_2|)}{| n_1 n_2|^{2a-1}\zeta(2a)^2y}e\left( (  n_1+  n_2 )x\right)\,\mathcal{I}(  n_1,  n_2; y)\,,
\end{align*}
where
\begin{equation}\label{eq:integral}\mathcal{I}(   n_1,  n_2; y):=\int_{\R^2}\frac{h_{a,\lambda}\left(\frac{r_1r_2+1}{r_2-r_1}\right)}{|r_2-r_1|^{2a}}e\left(-(  n_1 y r_1+  n_2 y r_2)\right)\, dr_1\, dr_2\end{equation}
is the key integral in this approach, as it determines the general Fourier modes via the formulas
\begin{align}
\widehat f_{0,0}(y)& = \Sigma^{0,0}(y)+2\widehat\Sigma^{ 0,1}_{0}(y)+\widehat\Sigma^{ 1,1}_{0,0}(y),\nonumber\\
\widehat f_{n,0}(y)&=\widehat f_{0,n}(y) = \widehat\Sigma^{ 0,1}_{n}(y)+\widehat\Sigma^{ 1,1}_{n,0}(y)\ \ \qquad  \ \ (\text{for }n\neq 0), \text{ and }\\
\sum_{n_1\neq 0, n} \widehat f_{n_1,n-n_1}(y)
&= \frac{1}{\zeta(2a)^2y}
\sum_{n_1\neq 0, n} \frac{\sigma_{2a-1}(| n_1|)\sigma_{2a-1}(| n-n_1|)}{|n_1(n-n_1)|^{2a-1}}\mathcal{I}(n_1,n-n_1; y)\nonumber\end{align}
as in \cite[(A.44)]{GMV2015}.

However, in \cite{GMV2015} Green, Miller and Vanhove were unable to compute the integral  $\mathcal{I}(\cdot,\cdot; y)$, instead resorting to guessing the solutions for the Fourier coefficients.  In the next two sections we explicitly compute (\ref{eq:integral}) for $a=3/2$, thereby completing the Poincar\'e series approach and providing a method for computing $\mathcal{I}$ for $a\in \frac{1}{2}\mathbb{Z}_{>1}$.


We conclude this section with a discussion of solutions to the differential equation (\ref{eq:DE2}).
 Euler's method of variation of parameters gives solutions to  (\ref{eq:DE2})  of the form
\begin{align}\label{eq:hab}
h_{a,\lambda }(u)\  = \
&P_{\lambda_-}(i u) \int _1^u\frac{i (1- \sqrt{4 \lambda+1} )\left(x^2+1\right)^{-a} Q_{\lambda_-}(i x)}{2 \lambda \left(P_{\lambda_+}(i x) Q_{\lambda_-}(i x)-P_{\lambda_-}(i x) Q_{\lambda_+}(i x)\right)}\,dx\nonumber\\
   & \ \ \ \ +\,Q_{\lambda_-}(i u) \int _1^u\frac{i (1- \sqrt{4 \lambda+1})\left(x^2+1\right)^{-a} P_{\lambda_-}(i x)}{2 \lambda
   \left(P_{\lambda_-}(i x) Q_{\lambda_+}(i x)-P_{\lambda_+}(i x) Q_{\lambda_-}(i x)\right)}\,dx\\
  &  \ \ \ \ +\,c_1 P_{\lambda_-}(i u)\,+\, c_2 Q_{\lambda_-}(i u),\nonumber
\end{align}
where $\lambda_+ := \frac{1}{2} \left(\sqrt{4 \lambda+1}+1\right)$, $\lambda_-:=\frac{1}{2} \left(\sqrt{4 \lambda+1}-1\right)$,
 and $P_n$ and $Q_n$ are Legendre functions of the first and second kind.  The last two terms in (\ref{eq:hab}) are solutions to the homogenous part of (\ref{eq:DE2}).
 When $n$ is an integer (the main case of interest in this paper) these integrals can be explicitly computed.  For example, this recovers the formula
 \begin{equation}\label{eq:h3/2}h_{\frac{3}{2},12}(u) = \frac{7+44u^2+40u^4}{3\sqrt{1+u^2}}-\frac{16}{3\pi}\left( \frac{4}{3}+ 5u^2+u(3+5u^2)\arctan(u)\right) \end{equation}
from \cite[(A.7)]{GMV2015}.
Note that the first term $ \frac{7+44u^2+40u^4}{3\sqrt{1+u^2}}$ already solves (\ref{eq:DE2}), but the second term is necessary to ensure the decay condition $h_{\frac{3}{2},12}(u)= O( |u|^{-3})$.

The rest of this section is devoted to some specific comments about cases with $\lambda$ is integral.

\subsection{Half-integral $a$}
When $\lambda_+$ and $\lambda_-$ are both nonnegative integers, the denominators of the integrands in  (\ref{eq:hab}) are constant
 because
$$P_{n}(ix) Q_{n-1}(ix)-P_{n-1}(ix) Q_{n}(ix)=\textstyle{\frac{1}{n}}\,, \ n\in\ \Z_{>0},$$
and one can compute the integrals in (\ref{eq:hab})  using trigonometric substitution. In that case $h_{a,\lambda}$ will have a form similar to that of (\ref{eq:h3/2}).  For example, when $a=n+\frac{1}{2}$ and $\lambda=(2n+1)(2n+2)$ for some $n\in\mathbb{Z}_{> 0}$,
\begin{align}\label{eq:habhalf}
h_{a,\lambda}(u)
    &= c_1 P_{2n+1}(i u)+c_2 Q_{2n+1}(i u)+Q_{2n+1}(i u) \sum_{k=0}^{ n}A_{k}\left((1+u^2)^{-k+\frac{1}{2}}-2^{-k+\frac{1}{2}}\right)\nonumber\\
    &+P_{2n+1}(i u)\left(\sum_{k=0}^{n}B_{k}\Big[\frac{\arctan u}{ (u^2+1)^{k-\frac{1}{2}} } -\frac{\pi}{   2^{k+\frac{3}{2}}}\Big]+\sum_{k=0}^{2n+1} C_{k}\left[\left({\frac{u^2}{{1+u^2}}}\right)^{\!\!k+\frac{1}{2}} \!\! -\frac{1}{2^{k+1/2}}\right] \right),
\end{align}
where $A_k,B_k$ and $C_k$ are absolute constants, and $c_1,c_2\in\C$ are arbitrary.
In more general situations where $a$ and $\lambda$ are decoupled (such as considered in \cite{CGPWW2021,DKS2022a}), the $\arctan$ term frequently persists.

Notice that the solution (\ref{eq:habhalf}) already has a form similar to that of (\ref{eq:h3/2}). Specifically, $h_{a,\lambda}(u) $ for $a\in\frac{1}{2}\Z_{>0}$ is a linear combination of $\arctan u$, polynomials in $u$, and polynomials in $\frac{u}{\sqrt{u^2+1}}$. This form allows us to use the method outlined in Section \ref{sec:int} to compute the associated integral $\mathcal{I}$ in (\ref{eq:integral}) when $a=n+\frac{1}{2}$ and $\lambda=(2n+1)(2n+2)$.

\subsection{Integral $a$}

The case where $a\in\mathbb{Z}$ for $a\geq 2$ results in two different forms of $h_{a,\lambda}$ depending on whether $a$ is even or odd. Choosing $\lambda=a(a-1)$ again forces the Wronskian in the denominators of (\ref{eq:hab}) to be a constant.
However, computing the integrals in (\ref{eq:hab}) yields different types of solutions (arising from different trigonometric identities) depending on whether $a$ is odd or even. This reflects the findings in \cite{DKS2022a} and \cite{DKS2022b} where the authors examined solutions for a broader class of eigenvalues $\lambda_s= s(s-1)$ with  $s \in \{2, 4, \dots, 2a-2\}$.

For specific choices of even values of $a$ and   $\lambda_s= s(s-1)$ with  $s \in \{2, 4, \dots, 2a-2\}$, computing the integrals in (\ref{eq:hab}) yields $h_{a,\lambda}(u)$ of a similar form to that when $a=n+\frac{1}{2}$ (i.e., a linear combination of polynomials in $u$,$\frac{1}{\sqrt{1+u^2}}$, and $\arctan u$) but simplifies significantly when considering the order of vanishing.
It appears from examples that in these cases, the $\arctan u$ factors are absent and  the integral $\mathcal{I}$ can  be computed directly, without resorting to the method outlined in Section \ref{sec:int}.
For instance, we have
\begin{equation}
h_{2,2}(u) = \frac{1}{4(1+u^2)}, \quad  h_{4,30}(u) = \frac{1}{36(1+u^2)^3},  \quad h_{6,56}(u) = \frac{49+17u^2}{3200(1+u^2)^5}.
\end{equation}
 It is also worth noting that, at least in examples, the solution $h_{a,\lambda_s}$ contains extra logarithmic factors for  $s>2a-2$ (which are absent when $s\leq 2a-2$) that would again make the solution resulting in more complicated integral transforms.

 When $a$ is odd and  $\lambda_s= s(s-1)$ with  $s \in \{2, 4, \dots, 2a-2\}$, the situation becomes more complicated. For certain eigenvalues, the solution is a rational function similar to the even case mentioned above.  For instance, we find
 \begin{equation}
h_{3,2}(u) = \frac{7+5u^2}{32(1+u^2)^2}.
\end{equation}
However, in other cases (such as $a=3$ and $\lambda = 6$), computing the integrals in (\ref{eq:hab}) results in the appearance of factors of $\arctan ^2 u$, making the computation of $\mathcal{I}$ more difficult.
Our method below appears to require some modifications in order to work; for example, the integration by parts method applied in the next section seems unable to handle the presence of these $\arctan ^2 u$ terms.

\section{The integral $\mathcal{I}$}\label{sec:int}
In this section we show how to explicitly compute  the integral (\ref{eq:integral}) for   $a=n+\frac{1}{2}$ and $\lambda=(2n+1)(2n+2)$, illustrated with the example of $(a,\lambda)=(\frac 32,12)$ studied in \cite{GMRV2010}. Recall that for this choice of $a$ and $\lambda$ the integral is
$$\mathcal{I}(  n_1,  n_2; y):=\int_{\R^2}\frac{h\left(\frac{r_1r_2+1}{r_2-r_1}\right)}{|r_2-r_1|^3}e\left(  n_1 y r_1+  n_2 y r_2\right)\, dr_1\, dr_2,$$
with $h$ given by (\ref{eq:h3/2}).
The main tool will be a seven-fold integration by parts (in general, a $(4a+1)$-fold integration by parts) which removes the $\arctan(u)$ term completely. To wit, define
  $$H(r_1,r_2):=\frac{h\left(\frac{r_1r_2+1}{r_2-r_1}\right)}{|r_2-r_1|^3}$$ and the differential operator
 \begin{equation}
 Df : = \frac{\partial}{\partial r_1}f+ \frac{\partial}{\partial r_2}f.
 \end{equation}
  Then
 $$D^7\left(e\left( -y(n_1 r_1+n_2  r_2)\right) \right) = 128i\cdot e\left( -y(n_1 r_1+n_2  r_2)\right)(n_1+n_2)^7\pi^7y^7,$$
and seven-fold integration by parts yields the formula
\begin{equation}
\mathcal{I}(n_1, n_2; y) =  -\frac{1}{128i(n_1+n_2)^7\pi^7y^7 } \int_{\R^2}D^7H(r_1,r_2)e\left( -y(n_1 r_1+n_2  r_2)\right)\,dr_1\,dr_2
\end{equation}
(there are no boundary terms because of the decay conditions).
To compute the integral we separate out terms in the integrand depending on whether or not they are smooth at $r_1=r_2$.  Namely,
we write
\begin{equation}\label{eq:128down}
-\frac{D^7H(r_1,r_2)}{128i(n_1+n_2)^7\pi^7} = T_1(r_1,r_2) + T_2(r_1,r_2)\,,
\end{equation}
where direct calculation shows
 \begin{align}
T_1(r_1,r_2) &=\frac{|r_1-r_2|}{(r_1^2+1)^{7}(r_2^2+1)^{7}}\,p_1(r_1,r_2)\\
T_2(r_1,r_2) &=
\frac{1}{ (r_1^2+1)^{15/2}(r_2^2+1)^{15/2}}\,p_2(r_1,r_2),
\end{align}
with  $p_1(r_1,r_2)$ and $p_2(r_1,r_2)$    explicitly-computable polynomials with coefficients that depend on $n_1$ and $n_2$, and examine their respective contributions to the total integral.  Indeed, $T_1$ and $p_1$ come from the second term in (\ref{eq:h3/2}), while $T_2$ and $p_2$ come from the first term in (\ref{eq:h3/2}). In particular   certain complicated terms such as the arctangent from (\ref{eq:h3/2}) disappear through this seven-fold differentiation.  We will next compute $\mathcal{I}(n_1, n_2; y)$ as the sum $\mathcal{I}_1(n_1, n_2; y)$ + $\mathcal{I}_2(n_1, n_2; y)$ of integrals according  the  decomposition (\ref{eq:128down}).


\subsection{The first term $T_1$}

Here
\begin{equation*}
\mathcal{I}_1(n_1,n_2,y):= \frac{1}{ y^7}\int_{\R^2}\frac{|r_1-r_2|\cdot p_1(r_1,r_2)}{ (r_1^2+1)^{7}(r_2^2+1)^{7}} \cdot e\left( -y(n_1 r_1+n_2  r_2)\right)\, dr_1\,dr_2\,.
\end{equation*}
  After changing variables $r_1\to r_1+r_2$ and then breaking up the integration into two ranges depending on the sign of $r_1$, we write
\begin{align*}
\mathcal{I}_1&(n_1,n_2,y)
=\int_{\R^2}\ \frac{ |r_1| \, p_1(r_1+r_2,r_2) \cdot e\left( -y   (n_1 (r_1+r_2)+n_2 r_2)\right)}{  y^7
   \left((r_1+r_2)^2+1\right)^7 \left(r_2^2+1\right)^7 } \,  dr_1\,dr_2\\
 & = \int_\R\int_0^\infty
\frac{  r_1 \,  p_1(r_1+r_2,r_2) \cdot e\left( -y   (n_1 (r_1+r_2)+n_2 r_2)\right)}{  y^7
   \left((r_1+r_2)^2+1\right)^7 \left(r_2^2+1\right)^7 } \,  dr_1\,dr_2\\
& \ \ \  \ \  + \int_\R\int_0^\infty
\frac{  r_1 \, p_1(-r_1+r_2,r_2) \cdot e\left( -y   (n_1 (-r_1+r_2)+n_2 r_2)\right)}{  y^7
   \left((-r_1+r_2)^2+1\right)^7 \left(r_2^2+1\right)^7 } \,  dr_1\,dr_2\,.
\end{align*}
Assume $n_1+n_2\neq 0$ and change the order of integration, so that the $r_2$-integrals can be computed by shifting the contours up or down depending on whether $n_1+n_2$ is negative or positive.  After picking up poles this results in an expression of the form
\begin{equation}\label{aftershiftingr2}
\mathcal{I}_1(n_1,n_2,y)
= N(|n_1+n_2|\pi y) \int_0^\infty  \sum_{k=3}^6\frac{  c_k \,e^{-2\pi |n_1+n_2|y}}{ (\pi y(n_1+n_2))^{k+1}}  A_k(r_1)\,r_1^{-k}\, dr_1\,,
\end{equation}
where $c_3=-\frac 23$, $c_4=4i$, $c_5=10$, $c_6=-10i$, $N(x)=8 x^3 + 24 x^2 + 30 x + 15$,
  and
$$A_k(r_1) \ = \  e^{-2\pi i n_1 r_1 y}+e^{-2\pi i n_2 r_1 y} +(-1)^{k+1}  e^{2\pi i n_1 r_1 y} +(-1)^{k+1}  e^{2\pi i n_2 r_1 y}
 .$$
Although the integral (\ref{aftershiftingr2}) converges, the individual integrals   $ \int_0^\infty  A_k(r_1)r_1^{-k}\, dr_1$ do not.  We instead  evaluate
\begin{align*}
\int_0^\infty  A_k(r_1)\,r_1^{-k}\,e^{-\epsilon r_1}\,r_1^s\, dr_1
\end{align*}
which is holomorphic for
 $\text{Re}(s)>6$ and $\text{Re}(\epsilon)>0$, and calculate the full sum by analytic continuation.  Namely, we multiply the integrand in (\ref{aftershiftingr2}) by $e^{-\varepsilon r_1}r_1^s$ to get
\begin{align*}
\mathcal{I}_1(n_1,n_2,y) \  & = \  \sum_{k=3}^6\frac{ c_k\,  e^{-2\pi |n_1+n_2|y}}{  (\pi y(n_1+n_2))^{k+1}}N(|n_1+n_2|\pi y) \int_0^\infty  A_k(r_1)\,r_1^{-k}\,e^{-\epsilon r_1}\,r_1^s\, dr_1\\
& =  \ e^{-2 \pi  y |n_1+n_2|} N(|n_1+n_2|\pi y)
   \sum_{k=3}^6 \frac{c_k\,\Gamma(s-k+1)}{(\pi y(n_1+n_2))^{k+1}} \, {\mathcal A}_{k}(s,\epsilon)\,,
\end{align*}
where
\begin{multline*}
{\mathcal A}_k(s, \epsilon) \ = \
(-1)^{k+1}(\epsilon-2\pi i n_1 y)^{-1-s+k} \ + \ (-1)^{k+1}(\epsilon-2\pi i n_2 y)^{-1-s+k} \\ +(\epsilon+2\pi i n_1 y)^{-1-s+k} \ + \ (\epsilon+2\pi i n_2 y)^{-1-s+k}. \qquad\qquad
\end{multline*}
These integrals can then be computed in terms of the standard $\Gamma$-integral, and although not all of them are   holomorphic at $s=\epsilon=0$, the full sum is and evaluates there to give
\begin{align*}
 \mathcal{I}_1(n_1,n_2,y)& = \frac{2 e^{-2 \pi  y |n_1+n_2|} }{45 \pi ^2
   y^2 (n_1+n_2)^7}N(|n_1+n_2|\pi y) \ \times \\
     & \ \   \ \Big[(n_1+n_2) \left(n_1^4+14 n_1^3 n_2-94 n_1^2 n_2^2+14 n_1 n_2^3+n_2^4\right) \ + \\
       & \ \ \ \   \  60
   n_1^2 n_2^2 (n_1-n_2) \log\left(\textstyle{\frac{|n_1|}{|n_2|}}\right)\Big].
\end{align*}
This indeed matches the homogeneous part of the solution found by Green, Miller and Vanhove in \cite{GMV2015} and given by the $\alpha_{n_1,n_2}$ term in \cite[(2.40)]{GMV2015}, and of course came from the contribution to the homogeneous part of the differential equation (\ref{eq:DE2}).

\subsection{The second term $T_2$}

We write  $p_2(r_1,r_2)$ as $\sum_{k,\ell=0}^{12}d(k,\ell)r_1^kr_2^{\ell}$ for some coefficients $d(k,\ell)$.
The integral
$$\mathcal{I}_2(n_1,n_2,y) \ := \  \frac{1}{y^7}\int_{\R^2}\frac{ p_2(r_1,r_2)\cdot e\left(-y (r_1n_1+r_2n_2)\right) }{ (1+r_1^2)^{15/2}(1+r_2^2)^{15/2}} dr_1\,dr_2 $$
can be evaluated using the function
\begin{equation}\label{eq:Int}
{\mathcal K}(\xi) \ := \ \int_{\R}\frac{1}{ (1+r^2)^{15/2}}e^{-2\pi i r\xi}\, dr  \ = \  \frac{256\pi^7|\xi|^7K_7(2\pi |\xi|)}{135135}
\end{equation}
and its derivatives, via the   formula
\begin{align}\label{eq:ederivative}\frac{\partial^k}{\partial\xi^k}\left(\frac{e(- r\xi)}{ (1+r^2)^{15/2}}\right) = \frac{(-2\pi i)^{k}r^k e(- r\xi)}{ (1+r^2)^{15/2}}\,.
\end{align}
Assume $n_1\neq 0$, $ n_2\neq 0$ and $n_1+n_2\neq 0$.
Letting $\xi_1=n_1y$ and $\xi_2=n_2y$, and applying  (\ref{eq:Int})--(\ref{eq:ederivative}) gives
\begin{align}\label{eq:I2intermsofpartials}
\mathcal{I}_2(n_1,n_2,y) & \ = \ \frac{1}{y^7}\int_{\R^2}\frac{p_2(r_1,r_2)\, e(- r_1\xi_1-r_2\xi_2)}{ (1+r_1^2)^{15/2}(1+r_2^2)^{15/2}}  dr_1\,dr_2 \\
& \ = \ \sum_{k,\ell=0}^{12}\frac{d(k,\ell)}{y^7} \cdot  \int_{\R}\frac{r_1^k\, e(- r_1\xi_1)}{ (1+r_1^2)^{15/2}} dr_1\cdot  \int_{\R}\frac{r_2^{\ell}\, e(-r_2\xi_2)}{ (1+r_2^2)^{15/2}} \,dr_2\\
& \ = \  \sum_{k,\ell=0}^{12}\frac{y^{-7}d(k,\ell) }{(-2\pi i)^{k+\ell} }\cdot \frac{\partial^k{\mathcal K}}{\partial \xi_1^k}(\xi_1) \cdot  \frac{\partial^{\ell}{\mathcal K}}{\partial \xi_2^\ell} (\xi_2)\,. 
\end{align}
Using the derivative formula $K'_\nu(z)=\frac{1}{2} (-K_{\nu -1}(z)-K_{\nu +1}(z))$ as well as the recurrence
$K_\nu(z) = K_{\nu-2}(z) +\frac{2(\nu-1)}{z} K_{\nu-1}(z)\,,$
all $K$-Bessel functions arising can be expressed in terms of combinations of $K_0$ and $K_1$.
Applying the change of variables $\xi_i\mapsto yn_i$ we find that (\ref{eq:I2intermsofpartials}) equals
\begin{multline}
 \frac{8 n_1 n_2}{3 y^5 (n_1+n_2)^7}\,\times \\
  \ \ \    \Big[\Big(n_1^4
   \left(16 \pi ^2 n_2^2 y^2-1\right)+n_1^3 \left(40 \pi ^2
   n_2^3 y^2-14 n_2\right)  +2 n_1^2 n_2^2 \left(8 \pi ^2 n_2^2
   y^2+47\right)\\
    \ \ \ \ \ \ \ \ \ \  \ \ \  \ \ \    \ \ \  \ \ \ \ \  -4 \pi ^2 n_1^5 n_2 y^2 - 2 n_1 n_2^3 \left(2
   \pi ^2 n_2^2 y^2+7\right) - n_2^4\Big)   y^5 (n_1+n_2)\\
    \ \ \ \ \ \ \ \ \ \  \ \ \   \ \ \ \ \ \ \ \ \ \  \ \   \ \ \ \ \ \    \times
  \text{sgn}(n_1 )
   \text{sgn}(n_2) K_1(2 \pi  | n_1 y| ) K_1(2 \pi  | n_2 y|
   )\\
    \ \ \ \ \     +\Big(-
\pi^2 y^2 (n_1+n_2)^2 \left(13 n_1 n_2^2+n_2^3-65
   n_1^2 n_2+19 n_1^3\right)
  +30   n_1^2
   (n_2-n_1)\Big)  \\
     \ \ \ \ \ \ \ \ \ \  \ \ \   \ \ \ \ \ \ \ \ \ \  \ \   \ \ \ \ \ \   \ \ \ \ \    \times \frac{n_2 y^4 }{\pi}  \text{sgn}(n_1) K_1(2 \pi  | n_1 y| )
   K_0(2 \pi  | n_2 y| )\\
    \ \ \ \ \     +\Big(-
\pi^2 y^2 (n_1+n_2)^2 \left(13 n_1^2 n_2+n_1^3-65
   n_1 n_2^2+19 n_2^3\right)
  +30   n_2^2
   (n_1-n_2)\Big)  \\
     \ \ \ \ \ \ \ \ \ \  \ \ \   \ \ \ \ \ \ \ \ \ \  \ \   \ \ \ \ \ \   \ \ \ \ \    \times \frac{n_1 y^4 }{\pi}  \text{sgn}(n_2) K_0(2 \pi  | n_1 y| )
   K_1(2 \pi  | n_2 y| )\\
       \ \ \ \ \  + \Big(5 n_1^2 \left(4 \pi ^2 n_2^2
   y^2-3\right)+8 \pi ^2 n_1^3 n_2 y^2-2 \pi ^2 n_1^4 y^2+2
   n_1 n_2 \left(4 \pi ^2 n_2^2 y^2+15\right)\\
    \ \ \ \ \ \ \ \ \ \  \ \ \  \ \ \    \ \ \  \ \ \ \ \  -n_2^2 \left(2
   \pi ^2 n_2^2 y^2+15\right)\Big) \,2 n_1 n_2 y^5
   (n_1+n_2)\\
    \ \ \ \ \ \ \ \ \ \  \ \ \   \ \ \ \ \ \ \ \ \ \  \ \   \ \ \ \ \ \  \  \ \ \ \ \    \times  K_0(2 \pi  | n_1 y| ) K_0(2 \pi  |
   n_2 y| )\Big].
\end{multline}
Our computation shows that $\mathcal{I}_2$ matches the principal part of the solution found in \cite[(2.35)]{GMV2015}, in regions (iii) and (iv) there.
This handles the general case of $n_1$ and $n_2$; the remaining (degenerate) cases can either be computed directly or achieved as limits of the above analysis.

\section{A conjecture of Chester, Green, Pufu, Wang and Wen}\label{sec:Danylo}

The solution (\ref{eq:GMVsol}) to the differential equation (\ref{eq:DE}) was computed in \cite{GMV2015} and above here in terms of instanton/anti-instanton pairs.  In particular, (\ref{eq:GMVsol}) involves a $K_{7/2}$-Bessel function term for the homogeneous solutions, which is manifestly nonzero for particular pairs $\alpha_{n_1,n_2}$ in \eqref{eq:alpha1}--\eqref{eq:alpha3}.   However, when summed together in (\ref{eq:alpha}) to give the Fourier modes of the solution itself, they give a rather complicated expression.  Chester, Green, Pufu, Wang and Wen have conjectured, based on ideas from the AdS-CFT correspondence and Yang-Mills theory, that this total sum actually  {\it vanishes} for all $n\neq 0$  \cite{CGPWW2021}, i.e., (\ref{eq:conj2}).

We present a mathematical calculation which explains the vanishing in \eqref{eq:conj2}. However, our explanation falls short of a proof because it involves manipulations of divergent series.\footnote{After this paper was written, a broader statement (containing (\ref{eq:conj2}) as a special case) was rigorously proven in \cite{FKLR}. Though our argument correctly deduces that $\sum_{n_1+n_2=n}\alpha_{n_1,n_2}=0$, \cite{FKLR} shows that generalizations of such sums may result in Fourier coefficients of Hecke eigenforms, and in particular can be nonzero.
Their proof avoids the divergent series here by appealing to the Gross-Zagier Holomorphic Projection Lemma.} We will use these to derive the equivalent statement
 \begin{equation}\sum_{\substack{m+n=r\\mn\neq 0}}\alpha_{m,n}=-\sum_{\substack{m+n=r\\mn= 0}}\alpha_{m,n} = -\alpha_{0,r}-\alpha_{r,0}\,, \ \ \forall r>0,\end{equation}
 the case for $r<0$ being equivalent by interchanging $r$ and $-r$.

First, the desired identity    (\ref{eq:conj2}) can be more succinctly stated in terms of
	\[\widetilde \alpha_{m,n}  = \frac{45|m+n|^{5/2}}{128\pi}\alpha_{m,n}
	\ =  \sigma_2(|m|)\sigma_2(|n|) \Big(g\Big(\frac{2m}{m+n}\Big)
	                              +g\Big(\frac{2n}{m+n}\Big)\Big)\, \ \ \text{for } m,n\neq 0,\]
where
	\begin{align*}
	g(x) =  60x\log|x|-60\log|x|-60+\frac{24}{x}+\frac{4}{x^2}\,.
	\end{align*}
In terms of $\widetilde \alpha_{m,n}$, note that \eqref{eq:alpha2} gives
$$\widetilde \alpha_{r,0} = \widetilde \alpha_{0,r}=\sigma_2(r)\big(r^2\zeta(2)+60\zeta'(-2)\big)$$ since $\zeta(2) = \frac{\pi^2}{6}$ and $\zeta'(-2)=-\frac{\zeta(3)}{4\pi^2}$.
Thus
\eqref{eq:conj2}
is equivalent to  \begin{equation}\label{eq:conj3}\displaystyle\sum_{\substack{m+n=r\\mn\neq 0}}\widetilde\alpha_{m,n}=-2\sigma_2(r)\big(r^2\zeta(2)+60\zeta'(-2)\big).\end{equation}
The rest of this paper is devoted to formally explaining why
	\begin{equation} \label{eq:identity}
A_r:=
\sum_{c,d>0}(cd)^2
	\sum_{\substack{ad-bc=r\\ ab\ne 0}} \Big(g\Big(\frac{2ad}{r}\Big)+g\Big(\frac{2bc}{-r}\Big)\Big)	\
= -2\sigma_2(r)\big(r^2\zeta(2)+60\zeta'(-2)\big),
	\end{equation}
which is equivalent to \eqref{eq:conj3} by reindexing the sum.
We stress at the outset that although parts of this argument can be made rigorous, appropriately and correctly normalizing all the divergences involved is beyond the scope of this paper.

Let
\begin{equation}\label{Trcd}
  T_r(c,d)\ := \!\!\sum_{\substack{ad-bc=r \\ ab\ne0}}g\Big(-\frac{2bc}{r}\Big)\,,
\end{equation}
so that
	\begin{equation} \label{eq:ar}
	A_r
	\ = \ 2 \sum_{c,d>0}c^2\,d^2\, T_r(c,d) \,.
	\end{equation}
To analyze $T_r(c,d)$ for fixed $c,d>0$ we pick a solution $a=a^*,b=b^*$ to $ad-bc=r$ with $|b^*|$ minimal.  The general solution to $ad-bc=r$ can be parameterized as
\begin{equation}\label{eq:astarm}
a=a^*+\frac{m}{(c,d)}c, \ b=b^*+\frac{m}{(c,d)}d, \qquad m\in\Z,
\end{equation}
 since $(a-a^*)\frac{d}{(c,d)}-(b-b^*)\frac{c}{(c,d)}=0$ forces these coefficients $a-a^*$ and $b-b^*$ of the coprime integers $\frac{d}{(c,d)}$ and $\frac{c}{(c,d)}$ to have that form.

Our next goal is to perform Poisson summation on $T_r(c,d)$ over $m\in \Z$ (in terms of the parametrization in (\ref{eq:astarm})), keeping in mind that the conditions that neither $a$ nor $b$ vanish may constrain $m$ and force some terms to be excluded.  Vanishing of $a$ (resp. $b$) in the equation $ad-bc=r$ forces $c$ (resp. $d$) to divide $r$.  We break into four cases:
\begin{enumerate}
  \item Both $c|r$ and $d|r$. In this case there is a solution to $ad-bc=r$ with $a=0$, as well as a solution with $b=0$.  Thus $b^*=0$,  hence $a^*=\frac{r}{d}$, implying that  we must exclude both the terms with $m=0$ and $m=-\frac{r(c,d)}{cd}$ (which is in fact an integer, as can be seen by considering the prime factorization of the divisors $c$ and $d$ of $r$).  The excluded term corresponding to this latter value of $m$ is $g(-\frac{2bc}{r})=g(2)$.
  \item $d|r$ but $c\nmid r$.  In this case again $b^*=0$, but there is no solution with $a=0$; thus only the term for $m=0$ is excluded.
  \item $c|r$ but $d\nmid r$. Only the term $g(2)$, corresponding to the value of $-\frac{2bc}{r}$ at the solution to $ad-bc=r$ with $a=0$, is omitted.
  \item Both $c\nmid r$ and $d\nmid r$.  Here neither $a$ nor $b$ ever vanishes and no $m\in\Z$ is excluded.
\end{enumerate}
For brevity of notation we shall set    $x_{c,d}:=-\frac{b^*(c,d)}{d}$ and $y_{c,d}:=\frac{2cd}{r(c,d)}$.

To carry out formal Poisson summation, we first recall the formal integral computation that the function
\begin{equation}\label{eq:powerfunction}
x\mapsto |x|^{s-1}\sgn(x)^\eta, \quad \eta\in\{0,1\},
\end{equation}
 has Fourier transform
 $$G_\eta(s)|r|^{-s}\operatorname{sign}(-r)^\eta,$$ where
 $$G_0(s)=2(2\pi)^{-s}\Gamma(s)\cos(\pi s/2) \quad \text{and} \quad G_1(s)=2i(2\pi)^{-s}\Gamma(s)\sin(\pi s/2).$$  Applying Poisson summation to this Fourier transform pair (and its $s$-derivatives, to allow for logarithms), then invoking the  customary calculus of treating  the Fourier transform at $r=0$ as zero (which is valid for $\text{Re}{(s)}\ll 0$), results in the identity
\begin{equation} \label{eq:sumg}
	\begin{split}
	\sum_{n\in \mathbb{Z}} g(y_{c,d}(n+x_{c,d}))  =&
	-\frac{30y_{c,d}}{\pi}\im \operatorname{Li}_2(e(x_{c,d}))-60\log|1-e(x_{c,d})|\\ & \ \ \ \ \ \ \ \ \ \ \ \
	-\frac{24\pi i}{y_{c,d}}\frac{1+e(x_{c,d})}{1-e(x_{c,d})}
	-\frac{16\pi^2}{y_{c,d}^2}\frac{e(x_{c,d})}{(e(x_{c,d})-1)^2},
	\end{split}
	\end{equation}
for $x_{c,d}\notin\Z$.
When $x_{c,d}\in\Z$, we instead use the limiting case
\begin{equation} \label{eq:sumg1}
	\sum_{n\ne 0} g(ny_{c,d})  \ =
	60\log\Big|\frac{y_{c,d}}{2\pi}\Big|+60+\frac{4\pi^2}{3y_{c,d}^2}\,,
	\end{equation}
which can be derived more directly by noting that $g(x)+g(-x)=8 \left(-15 \log (| x| )-15+\frac{1}{x^2}\right)$.

It follows that
\begin{align}\label{Trcdddivr}
	T_r(c,d) \  = \
	60\log\Big|\frac{cd}{r(c,d)\pi}\Big|+60+\frac{r^2\pi^2(c,d)^2}{3(cd)^2} -
	g(2)\,\delta_{c|r}
\end{align}
for $d|r$, and
\begin{equation}\label{Trcddndivr}
\aligned
	T_r(c,d) \ = \   -g(2)\,\delta_{c|r}
		-\frac{60cd}{r(c,d)\pi}\im \mathrm{Li}_2(e(x_{c,d}))-60\log|1-e(x_{c,d})|\\
		-\frac{12r(c,d)\pi i}{cd}\frac{1+e(x_{c,d})}{1-e(x_{c,d})}
		-\frac{4r^2(c,d)^2\pi^2}{c^2d^2}\frac{e(x_{c,d})}{(e(x_{c,d})-1)^2}
\endaligned
\end{equation}
for $d\nmid r$.

Next define the double Dirichlet series
	\[L_r(s,w) := \sum_{c,d\ge 1}c^{-s}d^{-w}T_r(c,d)
	  =    \sum_{d\ge 1}\frac{L_{d,r}(s)}{d^w}\,,\ \ \text{ where }\    L_{d,r}(s):=\sum_{c\ge 1}c^{-s}T_r(c,d)\,, \]
so that $A_r=2L_r(-2,-2)$.  We first formally compute $L_{d,r}(s)$ using  the following Lemma.

\begin{lemma}\label{lem:dirichletseries}
	For $d\ge 1$ and $\text{Re}(s)$ sufficiently large    one has
	\[\sum_{c\ge1}\frac{\log|(c,d)|}{c^s} = \zeta(s)\Big(\sum_{\ell|d}\Lambda(\ell)\ell^{-s}\Big)\]
	and
	\[\sum_{c\ge1}\frac{(c,d)^k}{c^{s+k}} =
	\zeta(s+k)\prod_{p|d}\Big(1+(1-p^{-k})(p^{-s}+\dots+p^{-v_p(d)s})\Big)\,,\]
	where $\Lambda$ is the von Mangoldt function
\begin{equation}\label{vonMangoldt}
\Lambda(n) = \begin{cases}
      \log p, & n=p^k\text{ for some integer } k\geq 1,\\
      0, & \text{otherwise,}
   \end{cases}
\end{equation}
	  and $v_p(n) $ denotes the valuation of $n$ at prime $p$ (i.e.,  the  exponent of $p$ in the prime factorization of $n$).  In particular, the product is equal to $d^k$ when $s+k=0$.
\end{lemma}
\begin{proof}
	The first formula directly follows from the identity $\log|(c,d)| = \sum_{\ell|(c,d)}\Lambda(\ell)$ and changing the order of summation in the resulting double sum. To prove the second formula, we note that the Dirichlet series on the left has multiplicative coefficients, and so it suffices to prove that
	\[\sum_{n\ge0}\frac{p^{k\min(n,v_p(d))}}{p^{n(s+k)}} =
	(1-p^{-s-k})^{-1}\Big(1+(1-p^{-k})(p^{-s}+\dots+p^{-v_p(d)s})\Big)\,,\]
	which is easily verified after multiplying both sides by $(1-p^{-s-k})$.
\end{proof}

From (\ref{Trcdddivr}) and the above Lemma, it follows that
\begin{equation*}L_{d,r}(-2) = -60\zeta'(-2)-r^2\zeta(2)-g(2)\sigma_{2}(r)\, \qquad\text{for } d|r,
\end{equation*}
since $\zeta'(s) =-\sum_{n\geq 2}\frac{\log(n)}{n^s}$ and $\zeta(-2)=0$.
The case when $d\nmid r$ is more complicated.  We claim
\begin{equation}\label{Ldr(-2)}L_{d,r}(-2) = -g(2)\sigma_{2}(r)\,\qquad\text{for } d\nmid r\end{equation}
from which it formally follows that
\[ \sum_{d\ge1}\frac{L_{d,r}(-2)}{d^w} =
	-g(2)\sigma_2(r)\zeta(w)-\sigma_{-w}(r)(r^2\zeta(2)+60\zeta'(-2)),\]
and hence (\ref{eq:identity}) by specializing $w=-2$.

To obtain  (\ref{Ldr(-2)}) we appeal to (\ref{Trcddndivr}) and set $c'=\frac{c}{(c,d)}$, $d'=\frac{d}{(c,d)}$,  and $r'=\frac{r}{(c,d)}$, recalling from (\ref{Trcd}) that $ad-bc=r$ and hence $(c,d)$ must divide $r$.  Since $x_{c,d} \equiv -\frac{b(c,d)}{d}\equiv -\frac{b}{d'} \Mod{1}$, we see that $x_{c,d}\equiv \frac{\alpha}{d'} \Mod{1}$  is  determined in terms of the unique solution $\alpha\Mod{d'}$ to the congruence $\alpha c'=r'\Mod{d'}$.
 Thus for
$d\nmid r$
 	\begin{align*}
	L_{d,r}(s) = & \sum_{\substack{0<c<d\\ (c,d)|r}}\Big[
		-\frac{60\im
		\mathrm{Li}_2(\ee(x_{c,d}))}{r(c,d)d^{s-2}\pi}\zeta\left(s-1,\frac cd\right)
		-\frac{60\log|1-\ee(x_{c,d})|}{d^s}\zeta\left(s,\frac cd\right)\\
		& -\frac{12r(c,d)\pi
		i}{d^{s+2}}\frac{1+\ee(x_{c,d})}{1-\ee(x_{c,d})}\zeta\left(s+1,\frac cd\right)
		-\frac{4r^2(c,d)^2\pi^2}{d^{s+4}}
		\frac{\ee(x_{c,d})}{(\ee(x_{c,d})-1)^2}\zeta\left(s+2,\frac cd\right)\Big]
		\,\\& -g(2)\sigma_{-s}(r),
	\end{align*}
in terms of the Hurwitz zeta-function $\zeta(s,x)=\sum_{n\ge 0}(n+x)^{-s}$.  The same formal application of Poisson summation to (\ref{eq:powerfunction}) as before shows that the Hurwitz zeta functions satisfy $\zeta(s,x)+(-1)^s\zeta(1-x)=0$ when $s\in \Z_{\le 0}$ and $0<x<1$.   When terms for $c=c_1$ and $c=d-c_1$ are grouped together and evaluated at $s=-2$, the Hurwitz zeta-functions  therefore combine together to give the
vanishing of the sum, so that only the last term $-g(2)\sigma_{2}(r)$ remains, i.e., (\ref{Ldr(-2)}) is established as claimed.

\subsection*{Acknowledgements}
We wish to thank Guillaume Bossard, Eric D'Hoker, Dorian Goldfeld, Michael B. Green, Henryk Iwaniec, Axel Kleinschmidt, Daniel Persson, Boris Pioline, Peter Sarnak, Pierre Vanhove, and Don Zagier for their helpful conversations.  We would also like to thank the anonymous referees for several cogent and helpful comments.
K. K-L. acknowledges support from NSF Grant number DMS-2001909, while S.D.M. acknowledges support from NSF Grants numbers DMS-1801417 and DMS-2101841.


\end{document}